\documentclass[11pt,reqno,tbtags]{amsart}
\usepackage{amssymb}
\usepackage{natbib}
\bibpunct[, ]{[}{]}{;}{n}{,}{,}

\numberwithin{equation}{section}
\allowdisplaybreaks

\newtheorem*{property*}{Property \csname @currentlabel\endcsname}

\newtheorem{theorem}{Theorem}[section]
\newtheorem{lemma}[theorem]{Lemma}

\theoremstyle{definition}
\newtheorem{example}[theorem]{Example}
\newtheorem{definition}[theorem]{Definition}

\newtheorem{remark}[theorem]{Remark}
\newtheorem*{remark*}{Remark}

\theoremstyle{remark}

\newenvironment{romenumerate}{\begin{enumerate}
 }{\end{enumerate}}

\newcounter{oldenumi}
{\setcounter{oldenumi}{\value{enumi}}
\begin{romenumerate} \setcounter{enumi}{\value{oldenumi}}}
{\end{romenumerate}}

\newcounter{thmenumerate}
\newenvironment{thmenumerate}
{\setcounter{thmenumerate}{0}%
 \def\item{\par
 \refstepcounter{thmenumerate}\textup{(\roman{thmenumerate})\enspace}}
}
{}

\newcounter{xenumerate}   

\newcommand\pfitem[1]{\par(#1):}
\newcommand\pfitemref[1]{\par\ref{#1}:}

\newcommand{\refT}[1]{Theorem~\ref{#1}}

\newcommand{\refL}[1]{Lemma~\ref{#1}}
\newcommand{\refR}[1]{Remark~\ref{#1}}
\newcommand{\refS}[1]{Section~\ref{#1}}
\newcommand{\refD}[1]{Definition~\ref{#1}}

\newcommand{\refE}[1]{Example~\ref{#1}}

\newcommand{\refand}[2]{\ref{#1} and~\ref{#2}}

\newcommand\marginal[1]{\marginpar{\raggedright\parindent=0pt\tiny #1}}

\begingroup
  \divide\count255 by 60
  \multiply\count255 by -60
  \advance\count255 by \time
  \ifnum \count255 < 10 \xdef\klockan{\the\count1.0\the\count255}
  \else\xdef\klockan{\the\count1.\the\count255}\fi
\endgroup

\newcommand{\sumi}{\sum_{i=1}^\infty}

\newcommand{\sumim}{\sum_{i=1}^m}
\newcommand{\sumiM}{\sum_{i=1}^M}
\newcommand{\sumioM}{\sum_{i=0}^M}
\newcommand{\sumiom}{\sum_{i=0}^m}
\newcommand{\sumin}{\sum_{i=1}^n}
\newcommand{\prodim}{\prod_{i=1}^m}

\newcommand\set[1]{\ensuremath{\{#1\}}}
\newcommand\bigset[1]{\ensuremath{\bigl\{#1\bigr\}}}
\newcommand\Bigset[1]{\ensuremath{\Bigl\{#1\Bigr\}}}

\newcommand\bigpar[1]{\bigl(#1\bigr)}
\newcommand\Bigpar[1]{\Bigl(#1\Bigr)}
\newcommand\biggpar[1]{\biggl(#1\biggr)}

\def\rompar(#1){\textup(#1\textup)}    

\newcommand\parfrac[2]{\Bigpar{\frac{#1}{#2}}}
\newcommand\parrfrac[2]{\biggpar{\frac{#1}{#2}}}

\def\xexp(#1){e^{#1}}
\newcommand\ceil[1]{\lceil#1\rceil}
\newcommand\floor[1]{\lfloor#1\rfloor}

\newcommand\ntoo{\ensuremath{{n\to\infty}}}

\newcommand\iid{i.i.d.\spacefactor=1000}    
\newcommand\ie{i.e.\spacefactor=1000}
\newcommand\eg{e.g.\spacefactor=1000}

\newcommand\cf{cf.\spacefactor=1000}
\newcommand{\as}{a.s.\spacefactor=1000}
\newcommand{\aex}{a.e.\spacefactor=1000}
\newcommand\whp{whp}

\newcommand{\tend}{\longrightarrow}

\newcommand\pto{\overset{\mathrm{p}}{\tend}}

\newcommand\eqd{\overset{\mathrm{d}}{=}}

\newcommand\bbN{\mathbb N}

\newcounter{CC} 
\newcounter{cc}

\newcommand\E{\operatorname{\mathbb E{}}}
\renewcommand\P{\operatorname{\mathbb P{}}}

\newcommand\ga{\alpha}

\newcommand\gd{\delta}

\newcommand\gG{\Gamma}

\newcommand\gl{\lambda}

\newcommand\gs{\sigma}

\newcommand\eps{\varepsilon}

\newcommand\cA{\mathcal A}

\newcommand\cC{\mathcal C}

\newcommand\cE{\mathcal E}
\newcommand\cF{\mathcal F}
\newcommand\cL{{\mathcal L}}
\newcommand\cN{\mathcal N}

\newcommand\cS{{\mathcal S}}

\newcommand\cU{{\mathcal U}}

\newcommand\tH{{\widetilde H}}

\def\[#1]{[\![#1]\!]}

\newcommand\qw{^{-1}}

\renewcommand{\=}{:=}

\newcommand\oi{[0,1]}

\newcommand\dd{\,\textup{d}}

\newcommand\upto{\uparrow}

\newcommand{\Lovasz}{Lov\'asz}

\newcommand{\cuq}{\overline{\cU}}

\newcommand{\cuoo}{\cU_\infty}

\newcommand{\gdcut}{\gd_\square}
\newcommand{\hdcut}{\widehat\gd_\square}
\newcommand{\rmdcut}{d_\square}

\newcommand{\exch}{exchangeable}

\newcommand{\ggw}{\gG_W}
\newcommand{\ggwx}{\gG_{\wx}}

\newcommand{\cuc}{\cU_{\mathsf{c}}}
\newcommand{\cucx}{\cuc'}
\newcommand{\agg}{\ensuremath{\ga\Gamma_1\oplus(1-\ga)\Gamma_2}}
\newcommand{\aggi}{\ensuremath{\ga\Gamma'\oplus(1-\ga)\Gamma''}}
\newcommand{\oplusM}{\bigoplus_{i=1}^M}
\newcommand{\oplusm}{\bigoplus_{i=1}^m}

\newcommand{\oplusom}{\bigoplus_{i=0}^m}
\newcommand{\oplusoM}{\bigoplus_{i=0}^M}
\newcommand{\oplusmy}{\bigoplus_{1}^m}

\newcommand{\oplusiiM}{\bigoplus_{i=2}^M}
\newcommand{\oplusMii}{\bigoplus_{i=1}^{M''}}
\newcommand{\oplusiiMii}{\bigoplus_{i=2}^{M''}}
\newcommand{\oplusooy}{\bigoplus_{1}^\infty}
\newcommand{\seq}[2]{(#1_i)_{1}^{#2}}
\newcommand{\seqq}[2]{(#1_i)_{i=1}^{#2}}
\newcommand{\noll}{\ensuremath{\mathbf0}}
\newcommand{\gnw}{G(n,W)}
\newcommand{\goow}{G(\infty,W)}
\newcommand{\gng}{G(n,\gG)}
\newcommand{\goog}{G(\infty,\gG)}
\newcommand{\googi}{G(\infty,\gG_i)}
\newcommand{\csmu}{\ensuremath{(\cS,\mu)}}
\newcommand{\csmui}{\ensuremath{(\cS_i,\mu_i)}}
\newcommand{\aaa}{\cA}
\newcommand{\aaaq}{\cA_1}
\newcommand{\aaap}{\cA_+}
\newcommand{\aaaz}{\cA^{\downarrow}}
\newcommand{\unionq}{\dot{\bigcup}}
\newcommand{\cci}{\cC_1}
\newcommand{\pig}{\Pi_\gG}
\newcommand{\tfdot}{\ensuremath{t(F,\cdot)}}

\newcommand{\gin}{G_{in}}
\newcommand{\gjn}{G_{jn}}
\newcommand{\ggxn}{\gG_{(n)}}
\newcommand{\nn}{\cN}
\newcommand{\nnx}{\nn_x}
\newcommand{\restr}[1]{|_{#1}}
\newcommand{\wx}{W^*}
\newcommand{\oil}{\ensuremath{(\oi,\gl)}}
\newcommand{\nnn}{[n]}
\newcommand{\qv}{\widetilde V}
\newcommand{\qg}{\widetilde G}
\newcommand{\ax}{\widehat A}
\newcommand{\cin}{\cC_1(n)}
\newcommand{\gxin}{G_1(n)}
\newcommand{\qgin}{\qg_i[n]}
\newcommand{\sumimi}{\sum_{i>M'}}

\newcommand\REM[1]{{\raggedright\texttt{[#1]}\par\marginal{XXX}}}

\newcommand\urladdrx[1]{{\urladdr{\def~{{\tiny$\sim$}}#1}}}

\begin{document}
\title
{Connectedness in graph limits}

\date{February 26, 2008; revised April 10, 2008} 

\author{Svante Janson}
\address{Department of Mathematics, Uppsala University, PO Box 480,
SE-751~06 Uppsala, Sweden}
\email{svante.janson@math.uu.se}
\urladdrx{http://www.math.uu.se/~svante/}

\subjclass[2000]{} 

\begin{abstract} 
We define direct sums and a corresponding notion of connectedness for
graph limits. Every graph limit has a unique decomposition as a
direct sum of connected components. As is well-known, graph limits may
be represented by symmetric functions on a probability space; there
are natural definitions of direct sums and connectedness for such
functions, and there is a perfect correspondence with the
corresponding properties of the graph limit.
Similarly, every graph limit determines an infinite random graph,
which is a.s.\ connected if and only if the graph limit is connected.
There are also characterizations in terms of the asymptotic size of
the largest component in the corresponding finite random graphs, and
of minimal cuts in sequences of graphs converging to a given limit.
\end{abstract}

\maketitle

\section{Introduction and main results}\label{S1}

A deep and beautiful
theory of limit objects of dense graphs has in recent years been created
by
\citet{LSz} and
\citet{BCLi,BCLii}, and further developed in a series of further papers by
these and other 
authors. 
(In particular, \citet{BBCR} contains some
ideas similar to the ones in the present paper.)
Some basic components of the theory are reviewed below and in \refS{Slimit},
following the presentation in \citet{SJ209}.

We let $\cU$ be the set of unlabelled graphs; as is explained in
\refS{Slimit}, this is embedded in a compact metric space $\cuq$,
and we let $\cuoo\=\cuq\setminus\cU$, the (compact) set of 
graph limits. 

We denote the vertex and edge sets of a graph $G$ by $E(G)$ and
$V(G)$, and let $v(G)\=|V(G)|$ and $e(G)\=|E(G)|$ be the numbers of
vertices and edges.
All graphs in this paper are simple (\ie, without multiple edges or
loops) and undirected; we further assume $0<v(G)<\infty$,
except sometimes when we explicitly consider infinite graphs or
(sub)graphs that may be without vertices.

A well-known basic property for graphs is connectedness. We write the
disjoint union of two graphs $G_1$ and $G_2$ as $G_1\oplus G_2$; thus
$V(G_1\oplus G_2)=V(G_1)\cup V(G_2)$ and $E(G_1\oplus G_2)=E(G_1)\cup
E(G_2)$, where we assume that $V(G_1)\cap V(G_2)=\emptyset$.
(We usually work with unlabelled graphs, which can be regarded as
equivalence classes of labelled graphs, and then $V(G_1)\cap
V(G_2)=\emptyset$ can always be assumed by relabelling the vertices
when necessary.)
A graph is \emph{connected} if and only if it cannot be written as $G_1\oplus
G_2$ for two graphs $G_1$ and $G_2$. It is well-known, and easy to see,
that every graph $G$ can be uniquely decomposed as a disjoint union
$G_1\oplus G_2\oplus\dots\oplus G_M$ of connected graphs, the
\emph{components} of $G$.

The first aim of the present paper is to define an analoguous notion
of direct sum of graph limits and to establish analogues of these standard
results for graphs.
The natural definition seems to be that of a weighted sum. In \refS{Sdef},
we define for any two graph limits $\Gamma_1$ and $\Gamma_2$ in $\cuoo$
and any $\ga\in\oi$ a graph limit $\ga\Gamma_1\oplus(1-\ga)\Gamma_2$.
One way these sums can be characterized is the following, which shows
the connection with disjoint unions of graphs, and further yields an
explanation of the weight $\ga$.

All unspecified limits below are as \ntoo.

\begin{theorem}
  \label{TA1}
Suppose that $G_n$ and $G_n'$, $n\ge1$, are graphs  that satisfy
 $v(G_n)\to\infty$,  $v(G_n')\to\infty$,
 $G_n\to\Gamma\in\cuoo$ and
 $G_n'\to\Gamma'\in\cuoo$.
If further $v(G_n)/\bigpar{v(G_n)+v(G'_n)}\to\ga\in\oi$, then
$G_n\oplus G_n'\to\ga\Gamma\oplus(1-\ga)\gG'$.
\end{theorem}

The proof is given in \refS{Sdef}, while other results stated in this
section are proven in \refS{Spf}.

We use the analogy with graphs to define connectedness for graph
limits.

\begin{definition}
  \label{Dconn}
A graph limit $\gG\in\cuoo$ is \emph{disconnected} if there exist
graph limits $\gG_1,\gG_2\in\cuoo$ and $\ga\in(0,1)$ such that
$\gG=\agg$.
If this is not the case, $\gG$ is \emph{connected}.
\end{definition}

It is not obvious that there are any connected graph limits at all,
but this follows from the theorems below; an explicit example is given
in \refE{Equasi}. 

\begin{remark}
  Graph limits are defined in a rather abstract way, by taking the
  completion of $\cU$ in a suitable metric, and they are not
  constructed as topological spaces, nor can they (as far as I know)
  be seen as topological spaces in any meaningful way (as graphs can).
Hence, the connectedness defined and studied here is not the usual
  topological notion.
\end{remark}

\begin{definition}\label{Dcomp}
  Let $\gG\in\cuoo$. A \emph{component} of $\gG$ is a connected graph
  limit $\gG_1\in\cuoo$ such that $\gG=\agg$ for some $\ga\in(0,1]$
  and $\gG_2\in\cuoo$.
\end{definition}

It is not obvious that components exist, but we will see that every
graph limit has at least one component, with one exception discussed in
\refR{R0}.

A graph limit may have an infinite number of components. (An example
is given in \refE{Einfty}.)
Hence, in order to exhibit a graph limit as the sum of its components,
we need to extend the direct sum to the case of an infinite number of
terms. This too is done in \refS{Sdef}, where we define the direct sum
$\oplusm\ga_i\gG_i$ for any finite or infinite sequence $\seqq\gG m$
of graph limits and corresponding weights $\seqq{\ga}m$ with
$\ga_i\ge0$ and $\sumim\ga_i\le1$.
Let $\aaa$ be the set of all such sequences
$\seqq{\ga}m$ with $0\le m\le\infty$, every $\ga_i\ge0$ and
$\sumim\ga_i\le1$.
Further, 
let $\aaap$ be the subset of $\aaa$ of such sequences with every $\ga_i>0$,
and let $\aaaq$ be the subset of $\aaa$ of sequences with $\sumim\ga_i=1$.
(Although we allow $\ga_i=0$ in the definition of
$\oplusm\ga_i\gG_i$, 
all such terms may be deleted without changing the direct sum, \cf{}
\refD{Dsum} and \eqref{t1b}, and thus there is no essential restriction
to consider $\aaap$ only.)

\begin{theorem}
  \label{Tcomp}
Every graph limit $\gG\in\cuoo$ can be written as a finite or infinite
direct sum $\oplusM\ga_i\gG_i$ for some $M$ with $0\le M\le\infty$ and
some sequences $\seqq\gG M$ and $\seqq\ga M$ with every $\gG_i\in\cuoo$
connected and $\seqq\ga M\in\aaap$.

The number $M$ and the sequences $\seq\gG M$ and $\seq\ga M$ are
uniquely determined by $\gG$, up to simultaneous permutations of
the two sequences.

The set of components of $\gG$ equals $\set{\gG_i}_{i=1}^M$, \ie, a
graph limit $\gG'$ is a component of $\gG$ if and only if $\gG'=\gG_i$
for some $i$.
\end{theorem}

We call the number $M$ in \refT{Tcomp} the \emph{number of components}
of $\gG$.

\begin{remark}
  There may be repetitions in the sequence $\seq\gG M$; hence the
  number of distinct components, \ie, the number of graph limits
  $\gG'$ that are components of $\gG$, is not necessarily equal to
  $M$.
(For example, take any connected $\gG_1$ and let
  $\gG=\frac12\gG_1\oplus\frac12\gG_1$.) 
\end{remark}

\begin{remark}
  \label{Rdeficient}
It may seem surprising that we allow $\sum_i\ga_i<1$ in the definition
of direct sums and in \refT{Tcomp}. However, such \emph{deficient}
direct sums are equivalent to \emph{complete} ones (\ie,
$\sum_i\ga_i=1$) as follows.
There is a special (and trivial) graph limit  $\noll\in\cuoo$, which is the
limit of the empty graphs $E_n$, see \refE{Enoll}, and for any
sequences $\seq\gG M$ and $\seq\ga M$ with $\ga_i\ge0$ and
$\sumiM\ga_i<1$, if we let $\ga_0\=1-\sumiM\ga_i$ and $\gG_0=\noll$,
then
$\oplusM\ga_i\gG_i=\oplusoM\ga_i\gG_i$, where the latter direct
sum is complete.

Conversely, any summand $\ga_i\noll$ in a direct sum may always be
deleted.

It is thus a matter of taste whether we want to allow deficient direct
sums or not, in the latter case instead allowing a term $\ga_0\noll$
in the decomposition in \refT{Tcomp}. We prefer the version above,
partly because $\noll$ is disconnected and thus not a component of $\gG$.
\end{remark}

\begin{remark}
  \label{R0}
We allow $M=0$ in \refT{Tcomp}, but this occurs only in a trivial
case:
$M=0\iff$ $\gG$ has no components
$\iff$ $\gG=\noll$.
\end{remark}

\begin{remark}
  \label{R1}
A graph limit $\gG$ is connected if and only if it
has $M=1$ and $\ga_1=1$ in \refT{Tcomp}. Note that a graph limit with
$M=1$ and $\ga_1<1$ is \emph{not} connected since it equals the direct
sum $\ga_1\gG_1\oplus(1-\ga_1)\noll$, \cf{} \refR{Rdeficient}.  
\end{remark}

\subsection{Connectedness in graphs and their limits}

It should be obvious that connectedness (or its opposite) of the graphs
$G_n$ in a convergent sequence does not say anything about the limit.
In fact, the convergence of $G_n$ to a limit $\gG$ is not sensitive to
addition or deletion of $o(v(G_n)^2)$ edges to/from $G_n$, and such
changes might create or destroy connectedness.
For example, assuming $G_n\to\gG$, we may add edges from
a given vertex, say 1, in $G_n$, to all other vertices, thus creating
connected graphs $G_n'$ with $G_n'\to\gG$. 
Similarly, we may delete all edges having
vertex 1 as an endpoint, thus making 1 isolated and creating
disconnected graphs $G_n''$ with $G_n''\to\gG$. 

Instead, the connectedness of the limit of a sequence of graphs is
connected to quantitative connectedness properties of the graphs, more
precisely the sizes of minimal cuts.
Given two subsets $V',V''$ of the vertex set $V(G)$ of a graph $G$,
we let $e(V',V'')=e_G(V',V'')$ 
denote the number 
$|\set{ij\in E(G):i\in  V',\,j\in V''}|$ of edges between the two sets.

\begin{theorem}
  \label{Tcut}
Suppose that $G_n$ are graphs with $v(G_n)\to\infty$ and that
$G_n\to\gG\in\cuoo$. 
\begin{romenumerate}
  \item
If\/ $\gG$ is connected, then for every $\gd>0$ there exists $\eps>0$
such that, for all large $n$, if $V(G_n)=V'\cup V''$ is a partition
with $|V'|,|V''|\ge\gd v(G_n)$, then $e(V',V'')\ge\eps v(G_n)^2$.
  \item
If\/ $\gG$ is disconnected, then there exists $\gd>0$ and, for all large
$n$, 
a partition $V(G_n)=V'\cup V''$ 
with $|V'|,|V''|\ge\gd v(G_n)$ such that $e(V',V'')=o\bigpar{v(G_n)^2}$.
\end{romenumerate}
\end{theorem}

\refT{Tcut}(ii) (in a version using kernels, \cf, \refT{TK1} below,
is given (although not as a formal theorem) by \citet{BBCR}.

\subsection{Kernels}

As was shown by \citet{LSz}, see \refS{Slimit} for details, every
graph limit may be represented by a \emph{kernel} $W$
on a probability space \csmu, \ie, a
symmetric measurable function $W:\cS\times\cS\to\oi$.
(Furthermore, \csmu{} may be, and often is, chosen as \oil,
where $\gl$ is Lebesgue measure, but we will not insist on that.)
We let $\ggw\in\cuoo$ denote the graph limit represented by $W$. Note
that the representation is not unique; different $W$, even on
\oil, may give the same $\ggw$; again see \refS{Slimit}.

\begin{remark}\label{Rgraphon}
  \citet{BCLi,BCLii} use the term \emph{graphon} for such functions
  $W$.
However, they also consider the graphons as the graph limits, thus
  identifying equivalent graphons. We prefer for our purposes to be more 
  specific, and in order to avoid possible confusion between the two
  uses of ``graphon'', we will not use this term here. As just said,
  we use instead  ``kernel'' for such functions $W$.
\end{remark}

There are natural definitions of connectedness and sums for kernels,
which, as we shall see in Theorems \refand{TK1}{TK2}, correspond
directly to the corresponding 
notions for graph limits. (The non-uniqueness of the representation
does not cause any complications in the results, although we have to
worry about it in some proofs.)

\begin{definition}\label{DWconn}
  A kernel $W$ on \csmu{} is \emph{disconnected} if 
either $W=0$ \aex{} on $\cS\times\cS$ or
there exists a
  subset $A\subset\cS$ with $0<\mu(A)<1$ such that $W=0$ \aex{} on
  $A\times(\cS\setminus A)$. Otherwise $W$ is \emph{connected}.
\end{definition}

The exceptional (and trivial) case $W=0$ \aex{} has to be treated
separately only 
when $\cS$ is an atom, \ie{} if $\mu(A)=0$ or $1$ for every
measurable $A$. 

The same properties of kernels
were defined and studied in \cite{SJ178}
(in a somewhat more general situation),
but there called \emph{reducible} and \emph{irreducible}; the same
terms were used in \cite{BBCR}, where these properties where studied
further (in a similar context as the present paper).
(The trivial case $W=0$ \aex{} was not treated separately in
\cite{SJ178}; 
this made no difference there, but was with hindsight perhaps a mistake.)

By the \emph{disjoint union} $\unionq\cS_i$
of sets $\cS_i$, we mean their usual
union if the sets are disjoint; if not, we first make them disjoint by
replacing $\cS_i$ by $\cS_i\times\set i$.

\begin{definition}\label{DWsum}
  Let $0\le m\le\infty$ and let $\seq{\ga} m\in\aaa$; 
further, for
  each $i$, let $W_i$ be a kernel on a probability space \csmui.
  \begin{romenumerate}
	\item
In the complete case, $\seq{\ga} m\in\aaaq$,
let $\cS$ be the disjoint union $\unionq\cS_i$,
let $\mu$ be the probability measure on $\cS$ given by
$\mu(A)\=\sumim\ga_i\mu_i(A\cap\cS_i)$, and let 
the direct sum $\oplusm\ga_i W_i$ be
the kernel $W$ on \csmu{} defined by
\begin{equation*}
  W(x,y)=
  \begin{cases}
W_i(x,y), & x,y\in\cS_i;
\\
0, & x\in\cS_i,\, y\in\cS_j \text{ with }i\neq j.
  \end{cases}
\end{equation*}
\item
In the deficient case, $\seq{\ga} m\in\aaa\setminus\aaaq$,
take any probability space $(\cS_0,\mu_0)$, let $W_0\=0$ (on
$\cS_0\times\cS_0$) and $\ga_0\=1-\sumim\ga_i$, and define
$\oplusm\ga_iW_i\=\oplusom\ga_iW_i$.
  \end{romenumerate}
\end{definition}

For the definition of $\oplusm\ga_i W_i$ in the deficient case, \cf{}
\refR{Rdeficient} and note that $\noll$ is represented by $W_0=0$ on
any $\csmu$.
Our definition does not specify $\cS_0$ and $\mu_0$ and is thus
formally not a 
proper definition, but any choice will do for our purposes and the
flexibility is convenient.

\begin{remark}
  If $\csmui=(\oi,\gl)$ for every $i$, it is natural to make linear
  changes of variables to replace $\cS_i$ by the interval
  $I_i\=[\gs_{i-1},\gs_i)$ of length $\ga_i$, where $\gs_i\=\sum_{j\le
  i}\ga_j$; note that $(I_i)_1^m$ form a partition of $[0,1)$ and we
  obtain $\csmu=([0,1),\gl)$ or, if we prefer, $(\oi,\gl)$.
\end{remark}

\begin{remark}
  The notation for deficient sums of kernels has to be used with some
  care
(in particular in the case $m=1$):
the summands 
$\ga_iW_i$ do not mean the usual product.
\end{remark}

\begin{theorem}
  \label{TK1}
Let the graph limit $\gG$ be represented by a kernel $W$ on a
probability space \csmu. Then $\gG$ is connected if and only if $W$ is.
\end{theorem}

As said above, the representation is not unique, but the theorem implies
that all representing kernels are connected or disconnected simultaneously.

\begin{example}
  \label{Equasi}
Let $\gG_p$ be the graph limit given by the constant kernel
$W(x,y)=p$, for some $p\in\oi$ and some \csmu. (The graph limit in this
case depends on $p$ only, as a consequence of \eqref{t} below, which
yields $t(F,\gG_p)=p^{e(F)}$.)
By \refT{TK1}, $\gG_p$ is connected for $p>0$. 
(These graph limits are the limits of quasi-random sequences of
graphs \cite{ChungGW:quasi}, see \cite{LSz}.)
\end{example}

\begin{theorem}\label{TK2}
    Let\/ $0\le m\le\infty$ and let 
$\seq{\gG} m$ be graph limits and
$\seq{\ga} m\in\aaa$.
Suppose that, for
  each $i$, the graph limit $\gG_i$ is represented by a kernel $W_i$
  on a probability space \csmui. 
Then $\oplusm\ga_i \gG_i$ is represented by $\oplusm\ga_iW_i$.
\end{theorem}

\subsection{Random graphs}
A graph limit $\gG\in\cuoo$ defines an infinite random graph $\goog$,
which has vertex set $\bbN=1,2,\dots$ and is uniquely determined in
the sense that its distribution is. (Again, see \refS{Slimit}.)
Taking the subgraph induced by $[n]\=\set{1,\dots,n}$ we obtain finite
  random graphs $\gng$, $n=1,2,\dots$. These random graphs have the
  property that $\gng\to\gG$ in $\cuq$ a.s.

For the infinite random graph $\goog$, connectedness is equivalent to
connectedness of $\gG$.

\begin{theorem}
  \label{TR1}
Let\/ $\gG\in\cuoo$.
\begin{romenumerate}
  \item
If\/ $\gG$ is connected, then $\goog$ is \as{} connected.
  \item
If\/ $\gG$ is disconnected, then $\goog$ is \as{} disconnected.
\end{romenumerate}
\end{theorem}

For the (finite) random graph $\gng$, 
we cannot expect that connectedness is determined by the connectedness
of $\gG$, not even asymptotically.
For example, it is easy to see that if $W=p>0$ is constant as in
\refE{Equasi}, then 
$\gng=G(n,p)$ is connected \whp{}, \ie, with probability tending to
1 as \ntoo,
(this is one of the earliest results in random graph theory
\cite{Gilbert}, \cite{Bollobas}); 
on the other hand,
if $W(x,y)=x^2y^2$ on \oil, then
$\gng$ \whp{} has many 
isolated vertices
(at least $n^{1/4}$, considering only those
with $X_i<1.1n^{-3/4}$ in the construction in \refS{Slimit}).
Instead, there is a corresponding result on the
asymptotic size of the largest component. 
Let $\cci(G)$ denote the largest component of a
graph $G$ (splitting ties arbitrarily).

\begin{theorem}
  \label{TR2}
Let\/ $\gG\in\cuoo$.
\begin{romenumerate}
  \item
$\gG$ is connected if and only if\/ $|\cci(\gng)|/n\pto1$ as \ntoo.
\item
More precisely, for every $\gG$,
$|\cci(\gng)|/n\pto\rho$ for some $\rho\in\oi$;
if\/ $\gG$ is connected, then $\rho=1$, while 
if\/ $\gG$ is disconnected, then $0\le\rho<1$. Furthermore, $\rho=0$ if
and only if\/ $\gG=\noll$ (in which case $\gng=E_n$, the empty graph).
In fact, if\/ $\gG$ 
has the decomposition $\oplusM\ga_i\gG_i$ into components as
in \refT{Tcomp},
then
$\rho=\max_i\set{\ga_i}$ (with $\rho=0$ if $M=0$). 
\end{romenumerate}
\end{theorem}

\begin{remark}
\citet{SJ178} 
study the component sizes in the much sparser random graphs with edge
probabilities $O(1/n)$ obtained from $G(n,\gG)$
by randomly deleting edges, keeping each edge only with probability
$c/n$ for some constant $c$ (the range where a component of order $n$
appears), and 
  \citet{BBCR} extend this to a deterministic sequence of graphs
  $G_n\to\gG$; these much more intricate results easily imply (i).
\end{remark}

We can also give a complete description of the components of $\goog$.
We let $G\restr V$, where $G$ is a graph and
$V\subseteq V(G)$, denote the induced
subgraph of $G$ with vertex set $V$. (We allow here the possibility that
$V=\emptyset$, when the induced subgraph has no vertices. Such
cases may be ignored in (ii) in the theorem below, since they only occur for
small $n$, and thus do not affect the limit.) 

\begin{theorem}\label{TGcomp}
  Let\/ $\gG\in\cuoo$.
Let the random infinite graph $\goog$ have components $G_j$,
$j=1,\dots,N$ (with $1\le N\le \infty$), listed in increasing order of
the smallest element, say, and let $V_j\=V(G_j)$ be the vertex set of
$G_j$.
Then, \as, the following hold for every component $G_j$.
\begin{romenumerate}
  \item\label{tga}
$G_j$ is either infinite or an isolated vertex. I.e.,
  $|V_j|=v(G_j)\in\set{1,\infty}$. 
\item\label{tgb}
If\/ $G_j$ is infinite, then $G_j\restr{V_j\cap[n]}\to\gG_j'$ as \ntoo{}
for some (random) $\gG'_j\in\cuoo$.
\item\label{tgc}
The asymptotic density $\nu_j\=\lim_\ntoo|V_j\cap\nnn|/n$ exists;
furthermore, if\/ $|V_j|=\infty$, then $\nu_j>0$ (and trivially
conversely).
\item\label{tgd}
Consider the (random, and finite or infinite) sequence 
$\set{(\gG'_j,\nu_j)}_{j=1}^N$ defined by \ref{tgb} and \ref{tgc}.
Let\/ $\gG$ have the decomposition $\oplusM\ga_i\gG_i$ into components as
in \refT{Tcomp}.
Then, the subsequence \set{(\gG'_j,\nu_j):\nu_j>0} equals a permutation
of $\set{(\gG_i,\ga_i)}_{i=1}^M$.
\item\label{tge}
The set\/ $V_0\=\bigcup_{|V_j|=1}V_j$
of all isolated vertices in $\goog$ has \as{} a density 
$\nu_0\=\lim_\ntoo|V_0\cap\nnn|/n$, and
$\nu_0=\ga_0\=1-\sumiM\ga_i$ with $\ga_i$ as in \ref{tgd}.
Furthermore, $V_0=\emptyset\iff\ga_0=0\iff\seq\ga M\in\aaaq$, \ie, the
direct sum $\oplusM\ga_i\gG_i$ is complete.
\end{romenumerate}
\end{theorem}

The components of the infinite random graph $\goog$ define a random
partition of $\bbN$. Since $\goog$ is exchangeable (\ie, its
distribution is invariant under permutations of $\bbN$, see further
\cite{SJ209}), this yields an exchangeable random partition 
$\pig$ of $\bbN$.
\citet{Kingman} has shown, see also \citet[Section 2.3]{Bertoin}, 
that the blocks of an exchangeable
random partition \as{} have asymptotic densities, 
so an exchangeable random partition
has a (generally random) sequence $(p_i)_1^ \infty$
of asymptotic densities of the blocks in the partition; 
we assume that these are ordered  in decreasing order 
(possibly ignoring or adding 0's), and thus
$(p_i)_1^ \infty\in\aaaz\=\bigset{(p_i)_1^\infty\in\aaa:p_1\ge
  p_2\ge\dots}$.
Moreover, 
the (distribution of) $(p_i)_1^ \infty$ determines the
(distribution of) the \exch{} random partition and every sequence 
$(p_i)_1^ \infty\in\aaaz$ corresponds to an \exch{} random partition of
$\bbN$; 
the \exch{} random partition can be constructed from
$(p_i)_1^ \infty$  by the paint-box construction
\cite[Section 2.3]{Bertoin}.

\begin{theorem}\label{TR3}
  Let the graph limit $\gG$ have the decomposition $\oplusM\ga_i\gG_i$
  into components as in \refT{Tcomp}. Then the 
sequence $(p_i)_1^ \infty$
of asymptotic densities of the \exch{}
partition $\pig$ equals \as{} $\seq\ga M$ arranged in decreasing order,
  and extended by $0$'s if $M<\infty$.
\end{theorem}

In this case the asymptotic densities thus are (\as) deterministic.

It is also interesting to consider the component of $\goog$ containing
a given vertex, which we by symmetry can take to be 1.

\begin{theorem}\label{TR4}
Let\/ $\gG\in\cuoo$, and suppose that 
$\gG$ has the decomposition $\oplusM\ga_i\gG_i$
  into components as in \refT{Tcomp}. 
\begin{romenumerate}
\item
Let\/ $G_1$ be the component of $\goog$ that contains vertex $1$.
Then $|G_1|$ is \as{} either $1$ or $\infty$, with
$\P(|G_1|=\infty)=\sumiM\ga_i$ and $\P(|G_1|=1)=\ga_0\=1-\sumiM\ga_i$.
\item
Let\/ $H$ be the infinite random graph with $V(H)=\bbN$ obtained from
$G_1$ by 
relabelling the vertices in increasing order when $|G_1|=\infty$, and
simply taking $H\=E_\infty$, the infinite graph on $\bbN$ with no
edges, when $|G_1|=1$.
Then $H$ is an \exch{} infinite random graph, and $H$ has the same
distribution as the mixture $\sumioM\ga_i\cL(\googi)$, where
$\gG_0=\noll$.
In other words, for any measurable set $A$ of infinite graphs on
$\bbN$, 
$\P(H\in A)=\sumioM \ga_i\P(\googi\in A)$.  
\end{romenumerate}
\end{theorem}

\begin{remark}\label{Rsparse}
  The graph limit theory has versions for bipartite graphs and
  directed graphs too, see \cite{SJ209}. We presume that the
  definitions and results in the present paper have analogues for
  these cases too, but we have not pursued this.
Furthermore, \citet{BR} have recently started to extend the theory to
  limits of sparse graphs. We do not know whether our results can be
  extended to that case.
\end{remark}


\section{Graph limits}\label{Slimit}

We summarize some basic facts about graph limits that we will use as follows
(using the notation of \cite{SJ209}), see \citet{LSz},
\citet{BCLi,BCLii} and \citet{SJ209} for details and further results.

If $F$ and $G$ are two graphs, let
$t(F,G)$ denote the probability that a random mapping $\phi:V(F)\to V(G)$
defines a graph homomorphism, \ie, that $\phi(v)\phi(w)\in E(G)$ when
$vw\in E(F)$. (By a random mapping we mean a mapping uniformly chosen
among all $v(G)^{v(F)}$ possible ones; the images of the vertices in
$F$ are thus independent and uniformly distributed over $V(G)$.)
The basic definition \cite{LSz,BCLi} is that a sequence $G_n$ of graphs
converges if $t(F,G_n)$ converges for every graph $F$;
we will use the version in \cite{SJ209} where we further assume
$v(G_n)\to\infty$.
More precisely, 
the (countable and discrete) set $\cU$
of all unlabelled graphs can be embedded as a dense subspace of a
compact metric space 
$\cuq$ such that a sequence 
$G_n\in\cU$ of graphs with $v(G_n)\to\infty$
converges in $\cuq$ to some limit $\gG\in\cuq$ if and only if $t(F,G_n)$
converges for every graph $F$. 
We let $\cuoo\=\cuq\setminus \cU$ be the
set of proper limit elements, and define $v(\gG)\=\infty$ for $\gG\in\cuoo$.
The functionals $t(F,\cdot)$
extend to continuous functions on $\cuq$, 
and an element $\gG\in\cuoo$ is determined by the numbers $t(F,\gG)$.
Hence,
$G_n\to\gG\in\cuoo$ if and
only if $v(G_n)\to v(\gG)=\infty$ and
$t(F,G_n)\to t(F,\gG)$ for every graph $F$.
(See \cite{BCLi,BCLii} for 
deep results giving
several other, equivalent, characterizations of
$G_n\to\gG$; for example the fact that $\cuq$ may be metrized by the
cut distance $\gdcut$ defined there.)

Graph limits $\gG\in\cuoo$ may be represented (non-uniquely)
by functions as follows \cite{LSz},
see also \cite{SJ209} for connections to the Aldous--Hoover
representation theory for exchangeable arrays \cite{Kallenberg:exch}.
(See \cite{Tao} and \cite{Austin} for related results.)

Let $(\cS,\mu)$ be an arbitrary probability space and let
$W:\cS\times\cS\to\oi$ be a kernel, \ie, a symmetric measurable function.
Let $X_1,X_2,\dots,$ be an \iid{} sequence of random elements of $\cS$
with common distribution $\mu$. Then there is a (unique) graph limit
$\ggw\in\cuoo$ such that, for every graph $F$, 
\begin{equation}\label{t}
  \begin{split}
  t(F,\ggw)
&=
\E \prod_{ij\in\E(F)} W(X_i,X_j) 
\\
&=
\int_{\cS^{v(F)}} \prod_{ij\in\E(F)} W(x_i,x_j) 
 \dd\mu(x_1)\dotsm \dd\mu(x_{v(F)}).	
  \end{split}
\end{equation}
Further,  for every $n\ge1$, let
$\gnw$ be 
the random graph with vertex set $[n]$
and edges 
obtained by, conditionally given $X_1,X_2,\dots, X_n$, 
taking an edge $ij$ with probability $W(X_i,X_j)$, (conditionally)
independently for all pairs $(i,j)$ with $i<j$.
Then the random graph $\gnw$ converges to $\ggw$ \as{} as \ntoo.

We may in this construction also take $n=\infty$,
with $[\infty]=\bbN$; this gives a random
infinite graph $\goow$. Note that $\gnw$ is the induced subgraph 
$\goow\restr{\nnn}$.

Every graph limit in $\cuoo$ equals $\ggw$ for some such kernel
$W:\cS\times\cS\to\oi$ on a suitable probability space $(\cS,\mu)$; in
fact \cite{LSz}, see also \cite{SJ209} and \cite{SJL5}, 
we can always choose $\cS=\oi$
equipped with Lebesgue measure $\gl$. (This choice of $(\cS,\mu)$ is the
standard choice, and often the only one considered, but we find it
useful in this paper to be more general.)
Note, however, that even if we restrict ourselves to $\oil$ only, the
representing function $W$ is in general not unique. 
It is trivial that replacing $W$ by $W'$ with $W'=W$
$\mu\times\mu$-\aex{} does not affect the integral in \eqref{t} and
thus not $\ggw$. 
It is equally obvious that
a measure-preserving change of variables will not affect $\gG$.
Moreover, this can be extended a little further,
and the complete characterization of functions $W$ yielding the same
$\ggw$ is rather subtle, see \cite{BR,BCLi,SJ209,Kallenberg:exch} for details.

In view of this non-uniqueness, we will thus distinguish between the
graph limits, being elements of $\cuoo$, and the functions $W$ that
represent them; \cf{} \refR{Rgraphon}.

Given a graph limit $\gG\in\cuoo$, it can, as just said, be represented
as $\ggw$ for some kernel $W$. Although $W$ is not unique, the
distribution of the random graph $\gnw$ is the same for all
representing $W$, for every $n\le\infty$; 
in fact, for every graph $F$ with vertex set $[k]$ and $k\le n$,
$\P(\gnw\supseteq F)=t(F,\gG)$.
Consequently, for every graph limit $\gG$
there is a well-defined random graph $\gng$ with $n$ vertices, for
every $n$ with $1\le n\le\infty$; this includes the case $n=\infty$
when $\goog$ is an infinite random graph.
(See further \citet[Theorem 7.1 and Corollary 5.4]{SJ209}; that paper
treats only the case $\cS=\oi$, but
the general case can be proved the same way or by first transferring to
$\oi$ as in \cite{SJL5}.) 

We thus have, for every graph $F$ with $V(F)=[k]$ where $1\le k\le n\le\infty$,
\begin{equation}
  \label{magnus}
\P(\gng\supseteq F)=t(F,\gG).
\end{equation}

If $W$ is a kernel representing $\gG$, then $\gnw\to\gG$ \as, as said above.
Hence, 
for every graph limit $\gG$,  
\begin{equation}
  \label{rex}
\gng\to\gG
\qquad
\text{\as, as \ntoo.}
\end{equation}
One corollary of this, or of \eqref{magnus}, is that
the distribution
of $\goog$ determines $\gG$: if $G(\infty,\gG_1)\eqd G(\infty,\gG_2)$
for two graph limits $\gG_1,\gG_2$, then $\gG_1=\gG_2$.

\begin{example}
  \label{Enoll}
Let $E_n$ be the empty graph with $V(E_n)=[n]$ and $E(E_n)=\emptyset$. Then
$t(F,E_n)=0$ for any $F$ with $e(F)>0$, while, as always, $t(F,E_n)=1$
when $e(F)=0$.
Hence the sequence $(E_n)_n$ converges in $\cuq$, and there is a graph
limit $\noll\in\cuoo$ such that $E_n\to\noll$; this graph limit is
characterized by
\begin{equation}
  t(F,\noll)=
  \begin{cases}
0,&e(F)\ge1,
\\
1,&e(F)=0.	
  \end{cases}
\end{equation}

The graph limit $\noll$ is represented by the kernel $W=0$ (on any probability
space). Hence, $G(n,\noll)=E_n$ \as, for every $n\le\infty$.
It is easily seen that 
$\noll=\ga\noll\oplus(1-\ga)\noll$ for any $\ga\in\oi$;
hence $\noll$ is disconnected.

We may call $\noll$ the \emph{empty} or \emph{trivial} graph limit;
nevertheless it is useful and important, as is seen in \refS{S1}.
\end{example}

\section{Connected test graphs suffice}\label{SFconn}

In \refS{Slimit}, we  defined graph limits and convergence to them using the
functionals $t(F,\cdot)$ where $F$ ranges over the set of all
(unlabelled) graphs. It turns out that it suffices to use connected
graphs $F$.
More precisely, 
let $\cuc\subset\cU$ be the set of all connected unlabelled graphs,
and let $\cucx\=\set{G\in\cuc:e(G)>0}=\cuc\setminus\set{K_1}$ be the
subset of unlabelled connected graphs with at least one edge.
As the next lemma shows, the functionals $t(F,\cdot)$ for $F\in\cucx$
are sufficient to characterize graph limits as well as convergence to
them.

\begin{lemma}
  \label{LUC}
A graph limit $\gG$ is uniquely determined by the numbers $t(F,\gG)$
for $F\in\cucx$.

Moreover, if $\gG_1,\gG_2,\dots\in\cuq$ is a sequence of graphs or
graph limits with $v(\gG_n)\to\infty$, 
and 
for every $F\in\cucx$,
$t(F,\gG_n)\to t_F$ as \ntoo{} for some number
$t_F\in\oi$, then $\gG_n\to\gG$, where $\gG\in\cuoo$ is the unique
graph limit with $t(F,\gG)=t_F$, $F\in\cucx$.
\end{lemma}

\begin{proof}
We begin by observing that if $F=\oplusmy F_i$, then
\begin{equation}
  \label{d1}
t(F,\gG)=\prodim t(F_i,\gG),
\qquad \gG\in\cuq;
\end{equation}
if $\gG=G\in\cU$, this follows directly from the definition of
$t(F,G)$, and the general case $\gG\in\cuq$ follows by continuity.

Now, suppose that $\gG,\gG'\in\cuoo$ and that $t(F,\gG)=t(F,\gG')$ for
all $F\in\cucx$. Since trivially $t(K_1,\gG)=1=t(K_1,\gG')$, the
equality holds for all $F\in\cuc$, and thus by decomposing a graph $F$
into components and \eqref{d1},
$t(F,\gG)=t(F,\gG')$ for every graph $F$, \ie, $\gG=\gG'$.

This proves the first statement. The second statement is an immediate
consequence by compactness and a well-known general argument: Since
$\cuq$ is compact, there exist at least subsequences of $(\gG_n)$ that
converge. If $\gG\in\cuq$ is the limit of such a subsequence, then
$v(\gG)=\lim v(\gG_n)=\infty$, so $\gG\in\cuoo$, and
further, $t(F,\gG_n)\to
t(F,\gG)$ along 
the subsequence for every $F$, so $t(F,\gG)=t_F$ for $F\in\cucx$.
The first part now shows that any two convergent subsequences have the
same limit, which in a compact space implies that the entire sequence
converges. 
\end{proof}

\begin{remark}
  It is an interesting and still very much open question to study when
  a subset of all unlabelled graphs is sufficient to determine
  graph limits and  convergence to them. \refL{LUC} gives one general
  result.
Other results are known in special cases. For example, if $\gG_p$ is the
  graph limit 
in \refE{Equasi}
determined by a kernel $W$ that is constant $p\in\oi$,
  or, equivalently, $t(F,\gG_p)=p^{e(F)}$ for every $F$, 
and $G_n$ is a sequence of graphs with $v(G_n)\to\infty$,
then it suffices that $t(F,G_n)\to t(F,\gG_p)$ for the two graphs
$F=K_2$ and $F=C_4$ in order that  $G_n\to\gG_p$ (which is equivalent to
the well-known property that $G_n$ is quasirandom) 
\cite{ChungGW:quasi}, \cite{LSz}; this is
generalized by \citet{LSos} to a larger class of graph limits where
a finite set of $F$ suffices.
Another example is given by restricting $G_n$ to threshold graphs; in
  this case it suffices to consider stars $F$ \cite{threshold}.
\end{remark}

\section{Direct sums}\label{Sdef}

We first see how the functionals \tfdot{} behave for direct sums of graphs.
This is easier for connected $F$.

\begin{lemma}
  \label{L1}
Suppose that $F$, $G_1$ and $G_2$ are graphs with $F$ connected. Then
\begin{equation*}
  t(F,G_1\oplus G_2)
=
\parfrac{v(G_1)}{v(G_1)+v(G_2)}^{v(F)}t(F,G_1)
+
\parfrac{v(G_2)}{v(G_1)+v(G_2)}^{v(F)}t(F,G_2).
\end{equation*}
\end{lemma}

\begin{proof}
  Since $F$ is connected, a mapping $\phi:V(F)\to V(G_1\oplus G_2)$ is a
  graph homomorphism $F\to G_1\oplus G_2$ if and only if $\phi$ maps
  $V(F)$ into either $V(G_1)$ or $V(G_2)$, and further $\phi$ is a graph
  homomorphism $F\to G_1$ or $F\to G_2$, respectively.
If $\phi$ is a uniformly random mapping $V(F)\to V(G_1\oplus G_2)$, 
and $j=1$ or 2,
then the   probability that $\phi$ maps $V(F)$ into $V(G_j)$ is
$\bigpar{v(G_j)/(v(G_1)+v(G_2))}^{v(F)}$,
and conditioned on this event, $\phi$ is a graph homomorphism 
$F\to  G_j$ with probability $t(F,G_j)$.
\end{proof}

We use this lemma both as an inspiration and as a tool to define
direct sums of graph limits.

\begin{theorem}
  \label{T1}
  \begin{thmenumerate}
\item
If\/ $\gG_1,\gG_2\in\cuoo$ and $0\le\ga\le1$, then there exists a unique
graph limit $\gG\in\cuoo$ such that
\begin{equation}
  \label{t1a}
t(F,\gG)
=
\ga^{v(F)}t(F,\gG_1)
+
(1-\ga)^{v(F)}t(F,\gG_2),
\qquad
F\in\cuc.
\end{equation}
\item
More generally, 
if $\seqq G m$, where $1\le m\le\infty$, is a finite or infinite
sequence of elements of $\cuoo$, and $\seqq\ga m\in\aaa$ is a sequence
of weights, then there exists a unique
graph limit $\gG\in\cuoo$ such that
\begin{equation}
  \label{t1b}
t(F,\gG)
=
\sumim
\ga_i^{v(F)}t(F,\gG_i),
\qquad
F\in\cucx.
\end{equation}
  \end{thmenumerate}
\end{theorem}

\begin{definition}
  \label{Dsum}
The graph limit $\gG$ determined by \eqref{t1a} is denoted
$\agg$. 
More generally, the graph limit $\gG$ determined by \eqref{t1b} is
denoted
$\oplusm\ga_i\gG_i$.
\end{definition}

\begin{example}
  \label{Einfty}
We can now construct disconnected graphs, even with infinitely many
components, as direct sums. For example, let 
$\gG=\oplusooy 2^{-i} \gG_{1/i}$, 
with $\gG_{1/i}$ the connected graph limit defined
in \refE{Equasi}.
\end{example}

\begin{proof}[Proof of Theorems \refand{T1}{TA1}]
  The uniqueness in \refT{T1} follows from \refL{LUC}.

For the existence, we give a proof based on taking limits of graphs,
which also proves \refT{TA1}.
(An alternative construction is given by kernels and \refT{TK2}.)

Assume first that $m<\infty$ and $\sumim\ga_i=1$. Assume further that
$\gin\in\cU$, $1\le i\le m$ and $n\ge1$, are such that, as \ntoo,
$v(\gin)\to\infty$ and
\begin{align}
  \gin&\to\gG_i, 
\qquad 1\le i\le m,
\label{d3a}
\\
\frac{v(\gin)}{\sum_j v(\gjn)}
&\to\ga_i, 
\qquad 1\le i\le m.
\label{d3b}
\end{align}
Note that such graphs $\gin$ always may be found. For example, by
\eqref{rex}, there exist $H_{in}\in\cU$ with $v(H_{in})=n$ and
$H_{in}\to\gG_i$ as \ntoo; we may then take
$\gin\=H_{i,\ceil{n\ga_i}}$ if $\ga_i>0$ and, \eg, 
$\gin\=H_{i,\floor{\log n}+1}$ if $\ga_i=0$.

\refL{L1} extends immediately (\eg{} by induction) to disjoint sums of
several graphs, which yields, for every $F\in \cuc$, using \eqref{d3a}
and \eqref{d3b},
\begin{equation*}
t\Bigpar{F,\oplusm\gin}
=
\sumim
\parrfrac{v(\gin)}{\sum_j v(\gjn)}^{v(F)}t(F,\gin)
\to
\sumim \ga_i^{v(F)}t(F,\gG_i).
\end{equation*}
This shows by \refL{LUC} that $\oplusm\gin\to\gG$ for some
$\gG\in\cuq$ that satisfies \eqref{t1b}.

The special case $m=2$ yields \eqref{t1a} and, 
taking $G_{1n}$ and $G_{2n}$ as the given $G_n$ and $G_n'$,
\refT{TA1}.

In general, we have shown the existence of $\gG=\oplusm\ga_i\gG_i$
whenever $m<\infty$ and $\sumim\ga_i=1$.

Next, assume $m<\infty$ and $\sumim\ga_i<1$.
Let $\ga_0\=1-\sumim\ga_i$ and let $\gG_0\=\noll$ be as in
\refE{Enoll}.
By the case just shown, $\gG=\sumiom\ga_i\gG_i$ exists in
$\cuoo$. However, $t(F,\gG_0)=0$ for $F\in\cucx$, so \eqref{t1b} shows
that $\gG=\sumim\ga_i\gG_i$ too.

Finally, if $m=\infty$, define $\ggxn:=\sumin\ga_i\gG_i$.
By \eqref{t1b}, 
as \ntoo,
\begin{equation}
t\bigpar{F,\ggxn}
=
\sumin \ga_i^{v(F)}t(F,\gG_i)
\to
\sumi \ga_i^{v(F)}t(F,\gG_i),
\qquad
F\in\cucx 
.
\end{equation}
Hence, using \refL{LUC},
$\ggxn\to\gG$ for some $\gG$ satisfying \eqref{t1b}, \ie,
$\gG=\sumi\ga_i\gG_i$. 
\end{proof}

\section{Remaining proofs}\label{Spf}

\begin{proof}[Proof of \refT{TK2}]
Let $W\=\oplusm\ga_i W_i$.
If $\sumim\ga_i<1$, we first rewrite this  as
a complete direct sum
$W=\oplusom\ga_i W_i$ by \refD{DWsum}(ii). 

For any $F\in\cucx$, 
it follows by \refD{DWsum} that 
the integrand in \eqref{t} is 
non-zero only if all $x_k$ belong to the same $\cS_i$, and further
$i\neq0$; moreover, each $i>0$ gives a contribution 
$\ga_i^{v(F)} t(F,\gG_{W_i})$.
Thus 
$t(F,\gG_W)
=
\sumim
\ga_i^{v(F)}t(F,\gG_{W_i})
$ when $F\in\cucx$,
which completes the proof by \eqref{t1b} and \refD{Dsum}.
\end{proof}

\begin{lemma}
  \label{L0}
Let\/ $W$ be a kernel on a probability space
$(\cS,\mu)=(\cS,\cF,\mu)$. Define an operator $\nn$ on the
$\gs$-algebra $\cF$ by
\begin{align*}
  \nn(A)&\=\Bigset{x:\int_A W(x,y) \dd\mu(y)>0},
\qquad A\in\cF.
\intertext{Further, for a point $x\in\cS$, define}
  \nnx&\=\set{y: W(x,y) >0}.
\end{align*}
If\/ $W$ is connected, the following holds.
\begin{romenumerate}
  \item\label{l0a}
For $\mu$-\aex{} $x\in\cS$, $\mu(\nnx)>0$.
  \item\label{l0b}
If $\mu(A)>0$, then $\mu(\nn(A))>0$.
  \item\label{l0c}
If $\nn(A)\subseteq A$, then $\mu(A)=0$ or $\mu(A)=1$.
  \item\label{l0d}
If $A_1,A_2,\ldots\in\cF$, then 
$\nn\bigpar{\bigcup_{n=1}^\infty A_n}
=\bigcup_{n=1}^\infty \nn(A_n)$.
  \item\label{l0e}
If $\mu(A)>0$, then $\mu\bigpar{\bigcup_{n=1}^\infty \nn^n(A)}=1$.
\end{romenumerate}
\end{lemma}

\begin{proof}
First, note that for every $A\in\cF$,
if $x\notin\nn(A)$, then $W(x,y)=0$ for \aex{} $y\in A$, and thus by
Fubini
$W=0$ \aex{} on $(\cS\setminus\nn(A))\times A$. Hence, by the symmetry
of $W$, 
\begin{equation}
  \label{white}
\text{$W=0$ \aex{} on $A\times(\cS\setminus\nn(A))$},
\qquad 
A\in\cF.
\end{equation}

Next,  suppose that $A\in\cF$ is a subset of $\cS$ such that
  $W=0$ \aex{} on $A\times\cS$.
If $\mu(A)=1$, then $W=0$ \aex{} on $\cS\times\cS$, which by
  \refD{DWconn} contradicts the assumption that $W$ is connected.
Furthermore, $W=0$ \aex{} on $A\times(\cS\setminus A)$, which if
  $0<\mu(A)<1$ again contradicts \refD{DWconn}.
Consequently:
\begin{equation}
  \label{kalmar}
\text{$A\in\cF$ and $W=0$ \aex{} on $A\times\cS$}
\implies \mu(A)=0.
\end{equation}

\pfitemref{l0a}
Let $B\=\set{x:\mu(\nnx)=0}$.
If $x\in B$, then $W(x,y)=0$ for \aex{} $y\in\cS$, and thus, by
Fubini, $W=0$ \aex{} on $B\times\cS$. Hence, \eqref{kalmar} yields
$\mu(B)=0$.

\pfitemref{l0b}
If $\mu(\nn(A))=0$,
then \eqref{white} implies $W=0$ \aex{} on $A\times\cS$, and thus 
\eqref{kalmar} yields $\mu(A)=0$.

\pfitemref{l0c}
By \eqref{white}, $W=0$ \aex{} on
$A\times(\cS\setminus\nn(A))\supseteq A\times(\cS\setminus A)$. 
If $0<\mu(A)<1$, this means by \refD{DWconn} that $W$ is
disconnected, a contradiction.

\pfitemref{l0d}
Clearly, for every $x$, 
\begin{equation*}
  \begin{split}
x\notin\nn\biggpar{\bigcup_{n=1}^\infty A_n}
&\iff
\text{$W(x,y)=0$ for \aex{} $y\in\bigcup_n A_n$}
\\&
\iff 
\text{for every $n$,
$W(x,y)=0$ for \aex{} $y\in A_n$}
\\&
\iff 
\text{for every $n$,
$x\notin \nn(A_n)$}.
  \end{split}
\end{equation*}

\pfitemref{l0e}
Let $\ax\=\bigcup_{n=1}^\infty \nn^n(A)$.
Then, by \ref{l0d}, $\nn(\ax)=
\bigcup_{n=1}^\infty \nn^{n+1}(A)\subseteq \ax$.
Further, by \ref{l0b}, $\mu(\ax)\ge\mu(\nn(A))>0$. Hence \ref{l0c}
implies $\mu(\ax)=1$.
\end{proof}

\begin{lemma}\label{LW1}
  If $W$ is a connected kernel, then $\goow$ is \as{} connected.
\end{lemma}
\begin{proof}
  Recall the construction of $\goow$ in \refS{Slimit} using an \iid{}
  sequence $(X_i)_1^\infty$; we now do this construction by adding one
  vertex $i$ at a time, each time randomly choosing first $X_i$ and
  then, for each $j<i$, whether $ji$ is an edge or not.

Let $x\=X_1$.
By \refL{L0}\ref{l0a}, 
\as{} $\mu(\nnx)>0$; we assume this in the
sequel. Then, by \refL{L0}\ref{l0e}, 
$\mu\bigpar{\bigcup_{n=1}^\infty \nn^n(\nnx)}=1$, and thus \as{}
$X_2\in\bigcup_{n=1}^\infty \nn^n(\nnx)$. We assume that this happens
and choose an $n$ (depending on $X_1$ and $X_2$) such that 
$X_2\in\nn^n(\nnx)$. 
Let $i_n\=2$ and $x_n\=X_2$. 

Assume first that $n>1$.
For each new vertex $i$, the 
probability that $X_i\in\nn^{n-1}(\nnx)$
and that there is an edge $i_ni$ equals
$\int_{\nn^{n-1}(\nnx)}W(x_n,y)\dd\mu(y)$, which is $>0$ because
$x_n\in\nn^n(\nnx)$. Hence, \as, there exists some such $i>i_n$; let
$i_{n-1}$ be the first such $i$ and let $x_{n-1}\=X_{i_{n-1}}$.

Repeating $n-1$ times, we  \as{} find vertices $2=i_n<i_{n-1}<\dots<i_1$
that are connected to a path by edges in $\goow$, and with
$x_1\=X_{i_1}\in\nn(\nnx)$. 

Finally, for each new vertex $i>i_1$, the 
probability that 
there are edges $i_1i$ and $1i$ equals
$\int_{\cS}W(x_1,y)W(x,y)\dd\mu(y),$
which is $>0$ because 
$x_1\in\nn(\nnx)$, and thus there is a set $A\subseteq\nnx$ of
positive measure with $W(x_1,y)>0$ for $y\in A$, and $W(x,y)>0$ too
for $y\in A$ by the definition of $\nnx$.
Consequently, \as{} there exists some such vertex $i$, which means
that there is a path $1,i,i_1,\dots, i_n=2$ in $\goow$.

We have shown that \as{} the vertices 1 and 2 can be connected by a path
in $\goow$. By symmetry (exchangeability), the same is true for any
given pair of vertices $i$ and $j$. Since there is only a countable
number of vertices, \as{} $\goow$ is connected.
\end{proof}

\begin{lemma}
  \label{L2}
If\/ $W$ is a disconnected kernel, then $\ggw$ is disconnected.
\end{lemma}

\begin{proof}
  If\/ $W=0$ \aex, then $\ggw=\noll$, which is disconnected, see
  \refE{Enoll}.

By \refD{DWconn}, the other possibility is that there exists
$A\subset\cS$ with $0<\mu(A)<1$ and $W=0$ \aex{} on
$A\times(\cS\setminus A)$.
Let $\cS_1\=A$ and $\cS_2\=\cS\setminus A$; let further, for $j=1,2$,
$\ga_j\=\mu(\cS_j)$ and $\mu_j\=\ga_j\qw \mu\restr{\cS_j}$,
and let $W_j\=W\restr{\cS_j\times\cS_j}$, considered as a
kernel on $(\cS_j,\mu_j)$. 
Then $\wx\=\ga_1W_1\oplus\ga_2 W_2$ is a kernel defined on \csmu{} and
$\wx=W$ a.e.
(More precisely, $\wx$ equals $W$ modified to be identically 0 on
$A\times(\cS\setminus A)$ and 
$(\cS\setminus A)\times A$.)
Hence, by \eqref{t}, $\ggw=\ggwx$. Consequently, using \refT{TK2},
\begin{equation*}
  \ggw=\ggwx = \gG_{\ga_1W_1\oplus\ga_2W_2}
=\ga_1\gG_{W_1}\oplus\ga_2\gG_{W_2},
\end{equation*}
and thus $\ggw$ is disconnected.
(Recall that $\ga_1=\mu(A)\in(0,1)$.)
\end{proof}

\begin{lemma}
  \label{L3}
If\/ $\gG\in\cuoo$ is disconnected, then $\goog$ is disconnected \as
\end{lemma}

\begin{proof}
  By \refD{Dconn}, $\gG=\agg$ where $0<\ga<1$.
Choose a kernel $W_j$, on a probability space $(\cS_j,\mu_j)$, that
represents $\gG_j$ ($j=1,2$) and assume as we may that $\cS_1$ and
$\cS_2$ are disjoint. 
By \refT{TK2}, $\gG$ is represented by the kernel 
$W\=\ga W_1\oplus(1-\ga)W_2$ on $\cS\=\cS_1\cup\cS_2$, and thus
\begin{equation*}
  \goog=\goow=G(\infty,\ga W_1\oplus(1-\ga)W_2).
\end{equation*}
However, it is evident from the construction of $\goow$ in
\refS{Slimit} that there are no edges between 
$V_1\=\set{i:X_i\in\cS_1}$ and
$V_2\=\set{i:X_i\in\cS_2}$, and that these sets are \as{} non-empty;
hence $\goog=\goow$ is \as{} disconnected.
\end{proof}

\begin{proof}[Proof of \refT{TK1}]
If $W$ is connected, then $\goog=\goow$ is \as{} connected by
\refL{LW1}, and thus by \refL{L3}, $\gG$ cannot be disconnected, \ie,
$\gG$ is connected.

Conversely, if $W$ is disconnected, then $\gG=\ggw$ is disconnected by
Lemma \ref{L2}.   
\end{proof}

\begin{proof}[Proof of \refT{TR1}]
  Let $W$ be a kernel representing $\gG$. If $\gG$ is connected, then
  $W$ is connected by \refT{TK1} and thus $\goog=\goow$ is \as{}
  connected by \refL{LW1}.

If $\gG$ is disconnected, then $\goog$ is \as{} disconnected by \refL{L3}.
\end{proof}

\begin{lemma}
  \label{LWcomp}
Let\/ $W$ be a kernel on a probability space \csmu.
Then there is a decomposition $\cS=\bigcup_{i=0}^M\cS_i$ into disjoint
measurable sets $\cS_i$, where $0\le M\le\infty$ and $\cS_0$ may be
empty but $\ga_i\=\mu(\cS_i)>0$ for $i\ge1$, such that if 
$W_i\=W\restr{\cS_i\times\cS_i}$ and $\mu_i\=\ga_i\qw\mu\restr{\cS_i}$
($i\ge1$), then $W_i$ is a connected kernel on $(\cS_i,\mu_i)$ 
for $i\ge1$,
and
$W=0$ \aex{} on $(\cS\times\cS)\setminus\bigcup_{i=1}^M(\cS_i\times\cS_i)$;
hence
$\oplusM\ga_iW_i=W$ a.e.
\end{lemma}

\begin{proof}
  This is \citet[Lemma 5.17]{SJ178} expressed in the present
  terminology and notation, except that the definitions there allow
  $W_i=0$ \aex{} on $\cS_i\times\cS_i$ for some $i\ge1$; however, any
  such term may be deleted and $\cS_i$ included in $\cS_0$.
\end{proof}

\begin{proof}[Proof of \refT{Tcomp}, existence]
 Represent $\gG$ by a kernel $W$. By \refL{LWcomp}, 
$W=\oplusM\ga_i W_i$ \aex{} for some sequences $\seq W M$ of connected
 kernels and $\seq\ga M\in\aaap$, and thus $\gG$ is represented by 
$\oplusM\ga_i W_i$ too.

Let $\gG_i\=\gG_{W_i}$, and note that $\gG_i$ is connected by
\refT{TK1}.
By \refT{TK2}, $\oplusM\ga_i\gG_i$ is represented by the same kernel 
$\oplusM\ga_i W_i$ as $\gG$, and thus 
$\gG=\oplusM\ga_i\gG_i$.
\end{proof}

The proof of uniqueness is more complicated because the kernel $W$ is
not unique, and consequently it is not enough to show uniqueness of
the decomposition of $W$ in \refL{LWcomp}. To avoid having to use the
deep and rather subtle criteria for equivalence of kernels, we use
instead the random infinite graph $\goog$, which is uniquely
determined by $\gG$ (in the sense that its distribution is
determined); we thus study this further first.

\begin{proof}[Proof of \refT{TGcomp}]
Let $\gG=\oplusM\ga_i\gG_i$ be a decomposition as in \refT{Tcomp};
recall that we have proved existence of this (but not yet uniqueness).  
Choose a kernel $W_i$, on a probability space $(\cS_i,\mu_i)$, that
represents $\gG_i$ ($i\ge1$).
By \refT{TK1}, $W_i$ is connected, and by \refT{TK2}, $\gG$ is
represented by
$W\=\oplusM\ga_iW_i$ on $\csmu$, where $\cS=\bigcup_0^\infty \cS_i$
(assuming as we may that the sets $\cS_i$ are disjoint, and allowing
$\cS_0=\emptyset$). 
In the construction of $\goog=\goow$ in \refS{Slimit}, let
$\qv_i\=\set{k:X_k\in\cS_i}$
and let $\qg_i$ be the induced subgraph $\goow\restr{\qv_i}$.
(We here allow graphs with empty vertex set.)
Since $W(x,y)=0$ when $x\in\cS_i$ and $y\in\cS_j$ with $i\neq j$,
there are no edges in $\goow$ between $\qv_i$ and $\qv_j$; thus
$\goow=\oplusoM\qg_i$.
Furthermore, $W(x,y)=0$ for $x,y\in\cS_0$ too, so $E(\qg_0)=\emptyset$
and all vertices in $\qv_0$ (if any) are isolated.

By the law of large numbers, \as{} each $\qv_i$ has an asymptotic density 
\begin{equation*}
\lim_\ntoo|\qv_i\cap\nnn|/n
=\P(X_1\in\cS_i)
=\mu(\cS_i)
=\ga_i,
\qquad i\ge0.
\end{equation*}
In particular, for $i\ge1$, since then $\ga_i>0$,
$|\qv_i|=\infty$ a.s.
Moreover,
the subsequence
$\set{X_k:k\in\qv_i}$ is a sequence of \iid{} elements of $\cS_i$ with
the distribution $\mu_i$; hence the induced subgraph $\qg_i$ ($i\ge1$)
equals (in distribution), if we relabel the vertices in increasing
order as $1,2,\dots$, the infinite random graph
$G(\infty,W_i)=G(\infty,\gG_i)$.
In particular, by \refT{TR1}, \as{} each $\qg_i$, $i\ge1$, is
connected.
Consequently, the components of $\goog=\goow$ are \as{} given by
$\qg_i$, $i\ge1$, and the vertices in $\qv_0$, the latter being the
isolated vertices of $\goow$.
Moreover, for $i\ge1$, with $n_i(n)\=|\qv_i\cap\nnn|$, as \ntoo{} and
thus $n_i(n)\to\infty$,
\begin{equation*}
  \qg_i\restr{\qv_i\cap\nnn}
=G\bigpar{n_i(n),W_i}
\to\gG_i.
\end{equation*}
We have shown that
\set{(G_j,V_j)} \as{} is a permutation of 
$$
\set{(\qg_i,\qv_i):i\ge1}
\cup
\set{(K_1(x),\set x):x\in \qv_0},
$$
where $K_1(x)$ denotes the graph with vertex set \set{x} (and thus no
edges).
The results follow from the results just proven for $\qg_i$ and $\qv_i$.
\end{proof}

\begin{proof}[Proof of \refT{Tcomp}, uniqueness and components]
By \refT{TGcomp}\ref{tgd}, the sequence $\set{(\gG_i,\ga_i)}_{i=1}^M$ 
is \as{} a  permutation of the sequence 
\set{(\gG'_j,\nu_j):\nu_j>0} constructed there from the components of
$\goog$; hence the sequence is determined by $\gG$ up to
permutation. (In particular, $M$ is determined as being \as{} the
number of infinite components in $\goog$.)

Since it follows from \eqref{t1b} that 
the direct sum operation for graph limits is
associative in the natural way, we have 
$\gG=\oplusM\ga_i\gG_i=\ga_1\gG_1\oplus \oplusiiM\ga_i\gG_i$, and thus
$\gG_1$ is a component of $\gG$, and similarly every $\gG_i$, $i\ge1$.

Conversely, if $\gG'$ is a component of $\gG$, then
$\gG=\ga\gG'\oplus(1-\ga)\gG''$ for some $\ga>0$ and some
$\gG''\in\cuoo$. Decomposing $\gG''=\oplusiiMii\ga_i''\gG_i''$ 
(by the existence part already proven and relabelling),
we find
(again by associativity) a decomposition
$\gG=\oplusMii\ga_i'\gG_i'$ with $\gG'_1=\gG'$,
$\gG'_i=\gG_i''$ for $i\ge2$, $\ga'_1=\ga$ and
$\ga'_i=(1-\ga)\ga''_i$, $i\ge2$; hence
all $\ga_i'>0$ and all
$\gG'_i$ are connected. By the uniqueness just proved, $\gG'=\gG'_1=\gG_i$
for some $i$.
\end{proof}

\begin{proof}[Proof of \refT{TR3}]
  A corollary of \refT{TGcomp}\ref{tgc}\ref{tgd}, since
  the partition is $\set{V_j}$.
\end{proof}

\begin{proof}[Proof of \refT{TR4}]
    Another corollary of \refT{TGcomp} and its proof:
In the notation above, $G_1=\qg_i$ and $H\eqd G(\infty,\gG_i)$ if
$X_1\in\cS_i$, $i\ge1$, and $G_1=K_1$ and $H=E_\infty=G(\infty,\gG_0)$
if $X_1\in\cS_0$; furthermore, $\P(X_1\in\cS_i)=\mu(\cS_i)=\ga_i$.
\end{proof}

\begin{lemma}
  \label{LTR2}
If\/ $\gG\in\cuoo$ is connected, then $|\cci(\gng)|/n\pto1$. 
\end{lemma}
\begin{proof}
  Let $\cE_n$ be the event that vertices 1 and 2 are connected by a
  path in $\gnw$. Then $\cE_n\upto\cE_\infty$ and $\P(\cE_\infty)=1$
  by \refT{TR1}, and thus $\P(\cE_n)\to1$.

Let $\cin:=\cci(\gng)$ and let $\gxin$ be the component of $\gnw$ that
contains vertex 1. Then $|\cin|\ge|\gxin|$ and, using the symmetry,
\begin{equation*}
  \E(n-|\cin|)
\le
  \E(n-|\gxin|)
=
(n-1)\P(\cE_n^c)
=
(n-1)(1-\P(\cE_n))
=o(n).
\end{equation*}
Thus $\E(1-|\cin|/n)\to0$. Hence $1-|\cin|/n\pto0$, \ie{}  $|\cin|/n\pto1$.
\end{proof}

\begin{proof}[Proof of \refT{TR2}]
Let $\qv_i$, $\qg_i$ and $n_i(n)$ be as in the proof of \refT{TGcomp},
and let $\qgin\=\qg_i\restr{\nnn}$. (We allow this subgraph to lack
vertices.) 
Then $\gng$ is \as{} the direct sum of $\qgin$, $i\ge1$, together with the
  isolated vertices in $\qv_0\cap\nnn$ (if any); hence
  \begin{equation}\label{julie}
|\cci(\gng)|
=
1\vee \max_{i\ge1}|\cci(\qgin)|. 	
  \end{equation}
(Note that $\qgin$ does not have to be connected.)

Since $\qg_i\eqd G(\infty,\gG_i)$ and $v(\qgin)=n_i(n)$ by the proof
of \refT{TGcomp}, we have, conditioned on $n_i(n)$,
$\qgin\eqd G(n_i(n),\gG_i)$, and thus by \refL{LTR2} applied to the
connected graph limit $\gG_i$ (and considering only $n$ with $n_i(n)\ge1$)
\begin{equation}\label{sjw}
  \frac{|\cci(\qgin)|}{n}
=
  \frac{|\cci(\qgin)|}{n_i(n)}\cdot \frac{n_i(n)}n
\pto
1\cdot\ga_i
=\ga_i.
\end{equation}

Let $\rho\=\max_i\ga_i$ (with $\rho=0$ if $M=0$) and let $\eps>0$. 
If $\rho>0$, choose $i$ such
that $\ga_i=\rho$; then \eqref{julie} and \eqref{sjw} yield 
\begin{equation}
  \label{fin}
|\cci(\gng)|/n
\ge 
|\cci(\qgin)|/n
>
\rho-\eps
\qquad\text{\whp}.
\end{equation}
(If $\rho=0$, this is trivial.)

On the other hand, for every $i$, $n_i(n)/n\pto\ga_i$, and thus 
\whp
\begin{equation}
  \label{kalle}
|\cci(\qgin)|\le v(\qgin)
=n_i(n)
< (\ga_i+\eps)n
\le (\rho+\eps)n.
\end{equation}
Choose $M'<\infty$ such that
$\sumimi\ga_i<\eps$.
(If $M<\infty$ we may simply take $M'=M$.)
Then, similarly, 
by the law of large numbers,
$\sumimi n_i(n)/n\pto\sumimi \ga_i<\eps$, and thus 
\whp, for all $i>M'$ simultaneously,
$$
|\cci(\qgin)|\le v(\qgin)=n_i(n) \le \sumimi n_i(n) < \eps n.
$$
Since \whp{} \eqref{kalle} holds for every $i\le M'$, we find that
\whp{} $|\cci(\qgin)|/n\le \rho+\eps$ for all $i$, and thus 
\eqref{julie} implies $|\cci(\gng)|/n\le \rho+\eps$ \whp.

Since $\eps>0$ is arbitrary, this and \eqref{fin} 
show that $|\cci(\gng)|/n\pto \rho$ as asserted. Furthermore, it is
clear that 
$\rho=1$ if and only if $M=1$ and $\ga_1=1$, which holds if and only
if $\gG$ is connected, \cf{} \refR{R1}.
Similarly,
$\rho=0$ if and only if $M=0$, which holds if and only if $\gG=\noll$,
\cf{} \refR{R0}.
\end{proof}

Finally we consider the minimal sizes of cuts, \refT{Tcut}.
This proof differs from the others in this section and does not use
kernels or $\goog$; instead it uses, not surprisingly,
the cut distance \cite{BCLi}.

\begin{proof}[Proof of \refT{Tcut}]
\pfitem{i}
Suppose that the conclusion fails; we will show that then $\gG$ is
disconnected.
Thus we assume that there exists $\gd>0$ such that for every $\eps>0$
there exists a subsequence $\cN_\eps\subseteq\bbN$ along which there
exists partitions
$V(G_n)=V'\cup V''$ with $|V'|,|V''|\ge\gd v(G_n)$
and $e(V',V'')\le\eps{ v(G_n)^2}$.
Choosing $n_k$ in the subsequence $\cN_{1/k}$, $k\ge1$, and such that
$n_k>n_{k-1}$ for $k\ge2$, we obtain a subsequence $(G_{n_k})$ which
we relabel as $(G_n)$; then for every $n$ there is a partition 
$V(G_n)=V'\cup V''$ with $|V'|,|V''|\ge\gd v(G_n)$
and $e(V',V'')=o\bigpar{ v(G_n)^2}$.

Let $G_n'\=G_n\restr{V'}$, 
$G_n''\=G_n\restr{V''}$, and $G_n^*\=G_n'\oplus G_n''$, \ie, $G_n$
with the edges between $V'$ and $V''$ deleted.
Then $V(G_n^*)=V(G_n)$ and 
\begin{equation}
  \label{arn}
|E(G_n^*)\Delta E(G_n)|=e(V',V'')=o\bigpar{v(G_n)^2}.
\end{equation}
It is easy to see that this implies, for any graph $F$,
\begin{equation*}
  |t(F,G^*_n)-t(F,G_n)|
\le
v(F)^2\frac{|E(G^*_n)\Delta E(G_n)|}{v(G_n)^2}
=o(1).
\end{equation*}
Since $G_n\to\gG$, \ie{} $t(F,G_n)\to t(F,\gG)$, we obtain
$G_n^*\to\gG$ in $\cuq$.
(Alternatively, \eqref{arn} immediately yields, in the notation of
\cite{BCLi}, $\gdcut(G^*_n,G_n)\le\rmdcut(G^*_n,G_n)=o(1)$, and thus
$G_n^*\to\gG$ by \cite[Theorem 2.6]{BCLi}.)

Note that $v(G_n'),v(G_n'')\to\infty$ and $\gd\le v(G_n')/v(G_n)\le
1-\gd$.
By the compactness of $\cuq$ and $[\gd,1-\gd]$, we may select a
subsequence such that, along the subsequence,
$G_n'\to\gG'$ and $G_n''\to\gG''$ for some $\gG',\gG''\in\cuoo$, and
further $v(G_n')/v(G_n)\to\ga\in[\gd,1-\gd]$.
By \refT{TA1}, we thus have (still along the subsequence)
\begin{equation*}
  G_n^*\=G_n'\oplus G_n''
\to \ga\gG'\oplus(1-\ga)\gG''.
\end{equation*}
Since also, as shown above, $G^*_n\to\gG$, we conclude that
$\gG=\ga\gG'\oplus(1-\ga)\gG''$, and thus $\gG$ is disconnected.
  
\pfitem{ii}
Suppose that $\gG$ is disconnected; then $\gG=\aggi$ for some
$\gG',\gG''\in\cuoo$.
Choose graphs $H_n',H_n''$ with $v(H_n')=v(H_n'')=n$ and
$H_n'\to\gG'$, $H_n''\to\gG''$ as \ntoo. Further, let
$n'(n)\=\floor{\ga v(G_n)}$ and $n''(n)\=v(G_n)-n'(n)$, and consider
only $n$ so large that $n'(n),n''(n)\ge1$.
Then, $H_n\= H'_{n'(n)}\oplus H''_{n''(n)}$ has the same number of
vertices as $G_n$ and, by \refT{TA1},
$H_n\to\aggi=\gG$.

It follows that the combined sequence $G_1,H_1,G_2,H_2,\dots$
converges (to $\gG$), and thus by \citet[Theorem 2.6]{BCLi}
$\gdcut(G_n,H_n)\to0$, where $\gdcut$ is the cut distance defined in
\cite{BCLi}. Moreover, by \cite[Theorem 2.3]{BCLi}, this implies
  $\hdcut(G_n,H_n)\to0$, where
  \begin{equation}\label{dncut}
	\begin{split}
\hdcut(G_n,H_n)
&\=\min_{\tH_n\cong H_n} \rmdcut(G_n,H_n)	
\\&
\=\min_{\tH_n\cong H_n} 
\max_{S,T\subseteq V(G_n)}\frac{|e_{G_n}(S,T)-e_{\tH_n}(S,T)|}{v(G_n)^2},	  
	\end{split}
  \end{equation}
taking the minima over $\tH_n\cong H_n$ with the same vertex set
$V(G_n)$ as $G_n$. Fix a $\tH_n$ that achives the minimum in \eqref{dncut}.
Since $H_n$ can be partitioned into two parts with $n'(n)$ and $n''(n)$
vertices and no edges in between, the same holds for $\tH_n$, and thus
there exists a partition $V(G_n)=V'\cup V''$ with
$|V'|/v(G_n)=n'(n)/v(G_n)\to\ga>0$,
$|V''|/v(G_n)=n''(n)/v(G_n)\to1-\ga>0$, $e_{\tH_n}(V',V'')=0$ and
thus, by \eqref{dncut},
$$
e_{G_n}(V',V'')\le \rmdcut(G_n,\tH_n)v(G_n)^2=o\bigpar{v(G_n)^2},
$$
which proves the result for any $\gd<\min(\ga,1-\ga)$.
\end{proof}

\newcommand\AAP{\emph{Adv. Appl. Probab.} }
\newcommand\JAP{\emph{J. Appl. Probab.} }
\newcommand\JAMS{\emph{J. \AMS} }
\newcommand\MAMS{\emph{Memoirs \AMS} }
\newcommand\PAMS{\emph{Proc. \AMS} }
\newcommand\TAMS{\emph{Trans. \AMS} }
\newcommand\AnnMS{\emph{Ann. Math. Statist.} }
\newcommand\AnnPr{\emph{Ann. Probab.} }
\newcommand\CPC{\emph{Combin. Probab. Comput.} }
\newcommand\JMAA{\emph{J. Math. Anal. Appl.} }
\newcommand\RSA{\emph{Random Struct. Alg.} }
\newcommand\ZW{\emph{Z. Wahrsch. Verw. Gebiete} }
\newcommand\DMTCS{\jour{Discr. Math. Theor. Comput. Sci.} }

\newcommand\AMS{Amer. Math. Soc.}
\newcommand\Springer{Springer-Verlag}
\newcommand\Wiley{Wiley}

\newcommand\vol{\textbf}
\newcommand\jour{\emph}
\newcommand\book{\emph}
\newcommand\inbook{\emph}
\def\no#1#2,{\unskip#2, no. #1,} 
\newcommand\toappear{\unskip, to appear}

\newcommand\webcite[1]{\hfil
   \penalty1000
\texttt{\def~{{\tiny$\sim$}}#1}\hfill\hfill}
\newcommand\webcitesvante{\webcite{http://www.math.uu.se/~svante/papers/}}
\newcommand\arxiv[1]{\webcite{http://arxiv.org/#1}}

\def\nobibitem#1\par{}

\end{document}